\documentclass[11pt,reqno]{amsart}
\usepackage{graphicx,psfrag,amsxtra,color}
\usepackage{amscd}
\usepackage{multirow}
\usepackage[all]{xy}
\usepackage{pinlabel}
\usepackage{hyperref}

\newcommand{\cptwo}{\C\textup{P}^2}

\newcommand{\N}{\mathbb N}
\newcommand{\Z}{\mathbb Z}

\newcommand{\T}{\mathcal T}

\newcommand{\C}{\mathbb C}

\renewcommand{\phi}{\varphi}
\newcommand{\id}{\operatorname{id}}

\newcommand{\spinc}{\ifmmode{\operatorname{Spin}^c}\else{$\operatorname{spin}^c$\ }\fi}

\newcommand{\diff}{\operatorname{Diff}}          

\newtheorem{theorem}{Theorem}[section]

\newtheorem{lemma}[theorem]{Lemma}
\newtheorem{proposition}[theorem]{Proposition}
\newtheorem{corollary}[theorem]{Corollary}
\newtheorem{thm}{Theorem}

\theoremstyle{definition}
\newtheorem{definition}[theorem]{Definition}

\newtheorem{example}[theorem]{Example}

\title{Absolutely exotic compact $4$-manifolds}
\author[Selman Akbulut]{Selman Akbulut${}^1$}
\address{Department of Mathematics, \newline\indent Michigan State 
University \newline\indent East Lansing, MI 48824}
\email{akbulut@math.msu.edu}
\author[Daniel Ruberman]{Daniel Ruberman${}^2$}
\address{Department of Mathematics, MS 050\newline\indent Brandeis
University \newline\indent Waltham, MA 02454}
\email{ruberman@brandeis.edu}
\thanks{\noindent ${}^1$Partially supported by NSF grant DMS 0905917 and NSF FRG Grant 1065827.  
${}^2$Partially supported by NSF Grant 1105234 and NSF FRG Grant 1065827
}

\begin{document}
\begin{abstract}
We show how to construct absolutely exotic smooth structures on compact $4$-manifolds with boundary, including contractible manifolds. In particular, we prove that any compact smooth 4-manifold W with boundary that admits a relatively exotic structure contains a pair of codimension-zero submanifolds homotopy equivalent to W that are absolutely exotic copies of each other.  In this context, {\em absolute} means that the exotic structure is not relative to a particular parameterization of the boundary. Our examples are constructed by modifying a relatively exotic manifold by adding an invertible homology cobordism along its boundary. Applying this technique to corks (contractible manifolds with a diffeomorphism of the boundary that does not extend to a diffeomorphism of the interior) gives examples of absolutely exotic smooth structures on contractible 4-manifolds.  \end{abstract}
\maketitle

\vspace{-.2in}

\section{Introduction}
One goal of $4$-dimensional topology is to find exotic smooth structures on the simplest of closed $4$-manifolds, such as $S^4$ and $\cptwo$.  Amongst manifolds with boundary, there are very simple exotic structures coming from the phenomenon of corks,  which are {\em relatively} exotic contractible manifolds discovered by the first-named author~\cite{akbulut:contractible}. More specifically, a cork is a compact smooth contractible manifold $W$ together with a diffeomorphism $f: \partial W \to \partial W$ which does not extend to a self-diffeomorphism of $W$, although it does extend to a self-homeomorphism $F: W\to W$.  This gives an exotic smooth structure on $W$ relative to its boundary, namely the pullback smooth structure by F.  This smooth structure is not {\em absolute}, in the sense that  it is diffeomorphic to  $W$ if we don't fix the identification of the boundary.  We will explain this distinction more precisely in Section~\ref{S:relative}.  

In this paper we construct absolutely exotic smoothings of compact $4$-manifolds with  boundary, from relative exotic smoothings.  

\begin{thm}\label{T:general}
If $W$ is a compact smooth $4$-manifold
and $F:W \to W$ a homeomorphism whose restriction to $M = \partial W$ is a diffeomorphism that does not extend to a self diffeomorphism of $W$.  Then $W$ contains a pair of smooth $4$-manifolds $V$ and $V'$ homotopy equivalent to $W$ with $\partial V \cong \partial V'$, such that $V$ and $V'$ are not diffeomorphic to each other. 
\end{thm}

By applying Theorem~\ref{T:general} to corks $(W,f)$,  we get absolutely exotic contractible manifolds, it also applies to {\em anti-corks} (which are relatively exotic manifolds homotopy equivalent to $S^1$~\cite{akbulut:zeeman, akbulut:book,akbulut:stable}). Until now the smallest known absolutely exotic manifold with boundary was homotopy equivalent to $S^{2}$, and was constructed as a $4$-ball with a single $2$-handle attached \cite{akbulut1991exotic}.

\begin{thm}\label{T:main}
There are compact contractible smooth $4$-manifolds $V$ and $V'$ with diffeomorphic boundaries, such that they are homeomorphic but not diffeomorphic to each other. Similarly, there are absolutely exotic smooth manifolds which are homotopy equivalent to $S^1$.
\end{thm}

The cork theorem ~\cite{curtis-freedman-hsiang-stong,matveyev:h-cobordism} implies that $V'$ is obtained from $V$ by a cork-twisting operation in the interior; that is how we will obtain our examples.  

\vspace{.05in}

 In a final section, we will extend the technique to show how the existence of infinitely many relatively exotic contractible $4$-manifolds implies the existence of infinitely many absolutely exotic ones.
\\[2ex]
\textbf{Acknowledgments:} We thank Chuck Livingston and Jeff Meier for a helpful exchange of emails, and Dave Auckly and Nikolai Saveliev for helpful comments.  We particularly appreciate the generous assistance of Nathan Dunfield in helping us with the computer verification of the properties of the knots and manifolds that are used in our construction.

\newpage

\subsection{Smoothings and markings of the boundary}  \label{S:relative}
The material discussed here is standard but we review it to fix our terminology.  
\begin{definition}\label{D:relative}
Let $W^{n+1}$ be a compact topological manifold with boundary, and let $M^n$ be a closed smooth manifold.  A marking of the boundary is a homeomorphism $j: M \to \partial W$.  A smoothing of $W$ relative to the marking $j$ is a smooth structure on $W$, so that $j$ is a diffeomorphism.  Two relative smoothings $(W,j)$ and $(W',j')$ are equivalent (relatively diffeomorphic) if there is a diffeomorphism $F: W \to W'$ with $F\circ j = j'$.
\end{definition}
\vspace*{-2ex}
$$
\xymatrix@R=.2pc {W \ar[rr]^F&& W'\\
\\
\\
\ar[uuu]^{\subseteq} && \ar[uuu]_{\subseteq}\\
\partial W&  \ar[l]_j M \ar[r]^{j'} & \partial W'
}
$$
\vspace*{.5ex}

In the terminology of current $3$-manifold topology~\cite{lipshitz-ozsvath-thurston:bordered-hf}, this notion is described under the name of bordered manifold.  As an example, a smooth structure on $W^{4}$ induces, in a canonical way, a relative smoothing with $M^{3} = \partial W^{4}$ and $j$ the identity.
By composition with $j$, the set of diffeomorphisms of $M$, up to isotopy (more precisely, pseudo-isotopy) acts on the set of relative diffeomorphism classes of manifolds with boundary $M$; if $M = \partial W$, this amounts to replacing $j = \id $ by an arbitrary self-diffeomorphism. Corks are relative smoothings in this sense; the {\em Mazur cork} shown in Figure~\ref{F:mazur} was shown to be relatively exotic (in different terminology) in~\cite{akbulut:contractible,akbulut:zeeman}.

\begin{figure}[htb]
\labellist
\small\hair 2pt
 \pinlabel {$\cong$}  at 250 130
 \pinlabel {$0$}  at 175 257
 \pinlabel {$0$}  at 468 238
 \pinlabel {$\tau$}  at 405 275
\endlabellist
\centering
\includegraphics[scale=0.4]{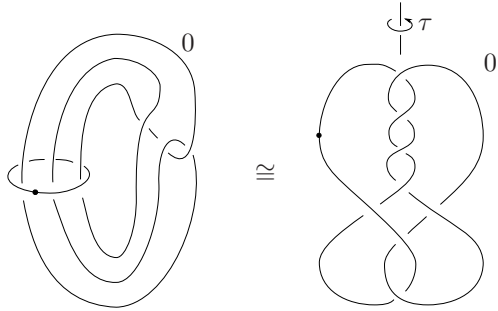}
\caption{The Mazur cork $(W,\tau)$}
\label{F:mazur}
\end{figure}

In contrast, an absolute smoothing of $W$ is just a smooth structure without a marking of the boundary, considered up to diffeomorphism.   If we are given a particular relative (resp.~absolute) smooth structure on $W$, then a relatively (resp.~absolutely) inequivalent smoothing will be referred to as {\em exotic}.  Sometimes there is no distinction between  relative and absolute, as in the following simple lemma.
\begin{lemma}\label{L:extend}
Suppose that every self-diffeomorphism of $\partial W$ extends to a diffeomorphism of $W$. Then the natural forgetful map from relative to absolute smoothings of $W$ is a bijection.
\end{lemma}
There are some well-known instances where this hypothesis is satisfied, for instance~\cite{cerf:diffS3} if $W =B^4$ or more generally~\cite{laudenbach-poenaru:handlebodies} if $W = \natural^n S^1 \times B^3$.  We will give another example as part of our main theorem.

\section{Constructing absolutely exotic $4$-manifolds}\label{S:construct}
The proof of Theorem~\ref{T:general} requires several ingredients from knot theory and $3$-dimensional topology.  We explain the basic idea first, and then show how to find those ingredients. We start with a standard definition.
\begin{definition}\label{D:invertible}
An invertible cobordism $X^{n+1}$ from $M^{n}$ to $N^{n}$  is a smooth manifold with $\partial X = - M \cup N$, such that there is a manifold $X'$ with 
$$
\partial X' = -N \cup M\ \text{and } X \cup_N X' \cong M \times I.
$$
\end{definition}
We will implicitly assume that there are markings of $\partial X$ and $\partial X'$ that are used in gluing $X$ to $X'$ along $N$, and that the diffeomorphism between $X \cup_N X'$ and $M \times I$ respects the markings of the $M$ boundary components.  It will be the case in our examples that the inclusions of $M$ and $N$ into $X$ induce isomorphisms on homology, so that $X$ is an invertible homology cobordism.  It is easy to see that the inclusion of $N$ into $X$ induces a surjection on the fundamental group. We will be exclusively concerned with $n=3$.

There is a relative version that leads to the construction of invertible cobordisms.
\begin{definition}\label{D:ds}
An invertible knot concordance~\cite{sumners:inv1,sumners:inv2} is an embedding $C \cong S^1 \times I \hookrightarrow S^3 \times I $ from $K_0$ to $K_1$ such that there is a concordance $C' \subset S^3 \times I$ with $C \cup C'$ isotopic to the product concordance $K_0 \times I$.
\end{definition}
We will only make use of the setting when $K_0$ is the unknot, in which case $K_1$ is {\em doubly slice}.  

\begin{lemma}\label{L:ds}
Suppose that $L = L_1,\ldots, L_n$ is a framed link in $M$, and that $C_i$ is an invertible concordance from the unknot to the knot $K_i$, $i=1\ldots n$. Define
\begin{equation}\label{E:ds}
X = M \times I - \left(L \times D^2 \times I\right) \bigcup \coprod_i \left(S^3 \times I - C_i \times D^2\right)
\end{equation}
where we glue the longitudes of each $K_i$ to the respective meridian of $L_i$ and vice versa.  Then $X$ is an invertible homology cobordism from $M$ to a $3$-manifold $N$.  If $\pi_1(S^3 \times I - C_i) \cong \Z$, then the inclusion $M \to X$ induces an isomorphism on fundamental groups.
\end{lemma}
\begin{proof}
That $X$ is a homology cobordism is standard~\cite{gordon:contractible}; the invertibility is obvious.  The statement about the fundamental group follows by van Kampen's theorem.
\end{proof}

We can now explain the basic idea of the proof of Theorems~\ref{T:main} and~\ref{T:general}.  Start with a relatively exotic manifold $(W,f)$ with $f$ the restriction of a homeomorphism $F:W\to W$. Form the union $V = W \cup_M X$, where $X$ is an invertible homology cobordism from $M =\partial W$ to some other $3$-manifold $N$.  We will construct $X$ using Lemma~\ref{L:ds}, so $V$ will be homotopy equivalent to $W$. Cutting out the embedded copy of $W$ in $V$ and regluing via $f$ results in a manifold $V'$, and the invertibility of $X$ will show that $V'$ is exotic relative to the identity marking on $\partial V' = N$.  To show that $V'$ is absolutely exotic, we will choose $N$ carefully so that all of its self-diffeomorphisms extend over $V'$.  Philosophically, we use the invertible homology cobordism to `kill' the symmetry of $M$.

\subsection{Taming the symmetry group of $M$}\label{S:symmetry}
Although our main interest lies in homology spheres, it turns out that the technique for modifying $M$ so that we understand its symmetry group is quite general.  The first step comes from a paper of Paoluzzi and Porti~\cite{paoluzzi-porti:links}
, who show that for any finite group $G$, there is a link $L$ in $S^3$ with hyperbolic complement, such that the symmetry group of $S^3- L$ is isomorphic to $G$. With minor modifications, their proof  works in an arbitrary $3$-manifold.
\begin{proposition}\label{P:G}
Let $M$ be an orientable $3$-manifold and $G$ be a finite group. Then there is a link $L$ in $M$ with hyperbolic complement, such that the symmetry group of $M- L$ is isomorphic to $G$.
\end{proposition}
\begin{proof}
The proof of the main theorem in~\cite{paoluzzi-porti:links} starts with a free and effective action of $G$ on an auxiliary $3$-manifold $A$ (called $M$ in~\cite{paoluzzi-porti:links}) and notes that $A$ can be viewed as surgery on G-invariant link $L$ in $S^3$, whose complement may be assumed to be hyperbolic via a result of Myers~\cite{myers1}.  The rest of the proof involves a further modification of the link by removing components lying in a standardly embedded genus-$2$ handlebody in its complement; the point of this is to make sure that the only symmetries of the complement are those given by the action of $G$.  Since any two $3$-manifolds are related by surgery on a framed link, it follows that 
$A$ could just as easily have been viewed as surgery on a link in $M$ with an effective free action of $G$ on its complement.  The link produced by the rest of the proof would then be a link in $M$ with the desired properties.
\end{proof}
Taking $G$ to be the trivial group, we get an obvious corollary.
\begin{corollary}\label{C:asymmetric}
Any orientable $3$-manifold $M$ contains a link $L$ with hyperbolic complement, such that the symmetry group of $M- L$ is trivial.
\end{corollary}
The link produced by the above proof would have $4$ components if $G$ is trivial.  In section~\ref{S:examples} we will give concrete examples of how to choose $L$; in those examples $L$ will in fact be a knot.  
\subsection{Some doubly slice knots}\label{ss:ds}

The other ingredient in our construction is a knot $J$ with the following properties: it is to be doubly slice, hyperbolic, and with trivial symmetry group.  We know of three such knots: the Kinoshita-Terasaka~\cite{kinoshita-terasaka} knot $11n42$ (Figure~\ref{figa3}) as well as $12n0313$, and $12n0430$.  These were found, starting with a list of doubly slice knots supplied by Jeff Meier, by a search on Knotinfo~\cite{knotinfo} and some computations with SnapPy~\cite{SnapPy}.   Such invariants are computed numerically, and in principle require a rigorous verification. Fortunately, the recent paper~\cite{dunfield-hoffman-licata:symmetry} shows how to certify the symmetry of certain 3-manifolds using interval arithmetic.  The arxiv listing for that paper contains code (based in turn on~\cite{HIKMOT}, which verifies hyperbolicity) that can be run, starting with a triangulation found via SnapPy, and will rigorously compute the symmetry group. All properties of the manifolds used in our construction were verified in this way; files describing the triangulations are available upon request to the authors.

We summarize the output of these calculations.
\begin{proposition}\label{P:calc}
The knots $11n42$, $12n0313$, and $12n0430$ are hyperbolic with trivial symmetry group and doubly slice, where the complement of each slice disk has fundamental group $\Z$.
\end{proposition}
\begin{proof}
The statements about hyperbolicity and symmetry were proved by computation, as described above.  We will show  $J = 11n42$ is doubly slice; this seems to be a well-known fact.  The other knots are left to the interested reader, as only $J$ is used in this paper. The dotted line in Figure~\ref{figa3} indicates the slice move for $J = 11n42$ (specifying a disk $D$ which $J$ bounds in $B^4$). 
\begin{figure}[ht]  \begin{center}
\includegraphics[width=.27\textwidth]{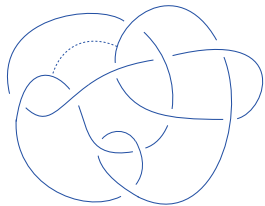}   
\caption{J (the dotted line indicates the slice move) }   \label{figa3}
\end{center}
\end{figure}

\begin{figure}[ht] 
\labellist
\small\hair 2pt
 \pinlabel {(a)}  at 120 -30
 \pinlabel {(b)}  at 420 -30
\endlabellist
 \begin{center}\includegraphics[width=.58\textwidth]{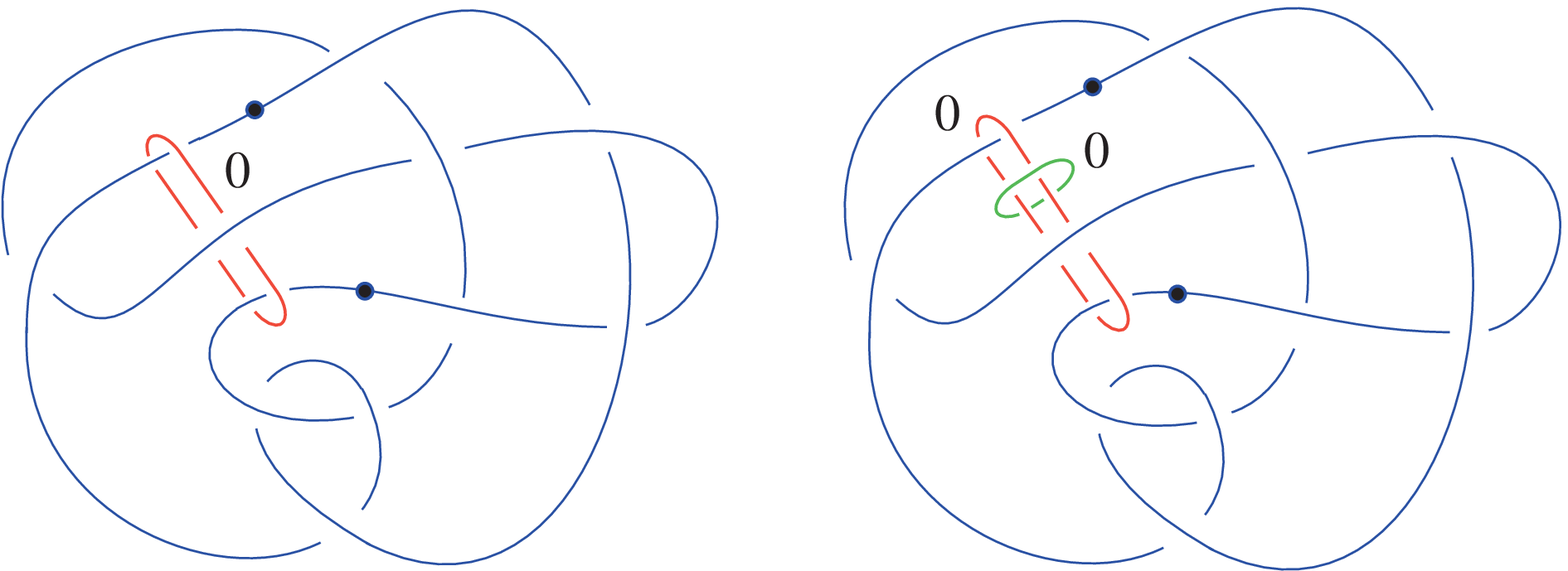}   
 \vspace*{1ex}
\caption{$B^{4}-D$ and $S^{4}-S^{2}$ }   \label{figa4}
\end{center}
\end{figure}
The first picture of Figure~\ref{figa4} is the handlebody of the complement of the disk $D$ in $B^{4}$, the second picture is the complement of the $S^2$ (which is the double of $D$) in $S^{4}$ (the reader can verify this by Section 1.4 of~\cite{akbulut:book}). After an isotopy, Figure~\ref{figa4} becomes Figure~\ref{figa5}.  The statement about the fundamental group comes from Figure~\ref{figa5}(a), because the $2$-handle algebraically cancels the lower dotted $1$-handle. The double sliceness of $J$ is now evident in Figure~\ref{figa5}(b), i.e. the complement of  $S^{2}$ in $S^{4}$ is $S^1 \times B^3$, so $S$ is an unknotted $2$-sphere~\cite{gluck:twist}.
\end{proof}
\begin{figure}[ht]  
\labellist
\small\hair 2pt
 \pinlabel {(a)}  at 84 -30
 \pinlabel {(b)}  at 345 -30
\endlabellist
\begin{center}
\includegraphics[width=.75\textwidth]{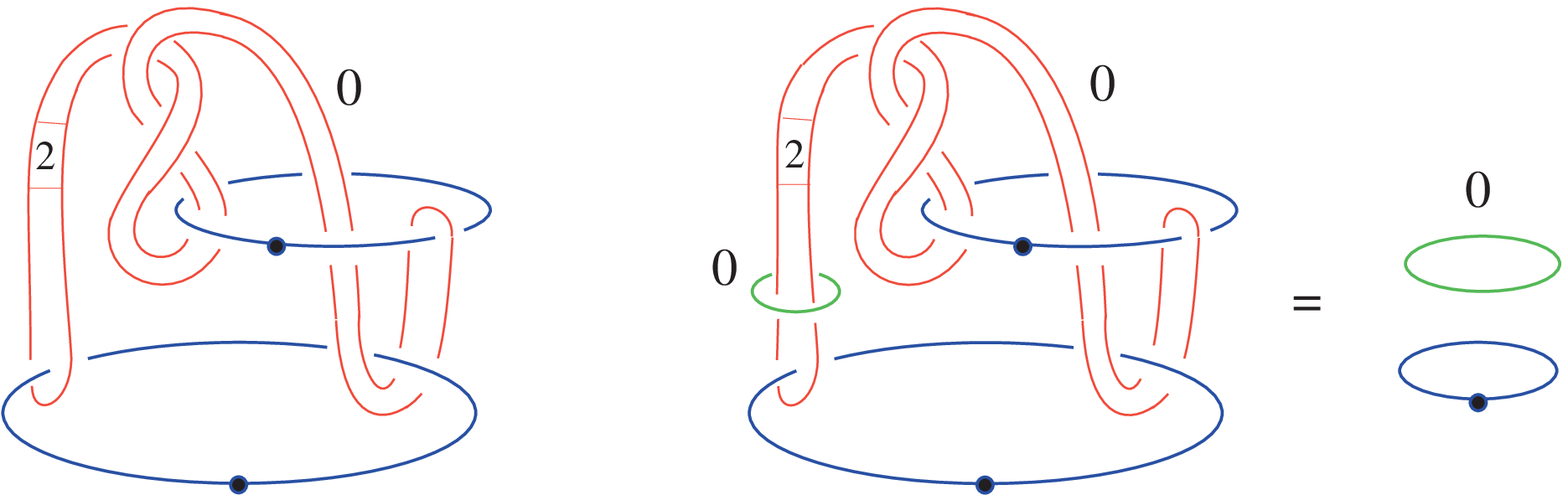}   
 \vspace*{1ex}
\caption{$B^{4}-D$ and $S^{4}-S^{2}$  }   \label{figa5}
\end{center}
\end{figure}

%
%
%
\section{Proof of Theorems~\ref{T:general} and~\ref{T:main}}\label{S:main-proof}
We assemble the results from the previous sections to prove Theorem~\ref{T:general}; afterwards we will prove Theorem~\ref{T:main}. 
\begin{proof}[Proof of Theorem~\ref{T:general}]
Recall the setup; we have a $4$-manifold $W$ with $\partial W = M$, and a homeomorphism $F$ whose restriction $f$ to $M$ is a diffeomorphism that does not extend to a diffeomorphism of $W$.   Using Corollary~\ref{C:asymmetric}, choose an $n$-component link $L \subset M$ so that $M-L$ is hyperbolic and has trivial symmetry group.  Let $C$ be the invertible concordance from the unknot $U$ to $J = 11n42$. Then form the homology cobordism $X$  from $M$ to $N$ described in Lemma~\ref{L:ds}.   We showed in Proposition~\ref{P:calc} that $\pi_1(S^3 \times I - C) \cong \Z$, so that $\pi_1(M) \to \pi_1(X)$ is an isomorphism.

We will use these facts to verify the following.\\[1ex]
{\bf Claim:} The group of diffeomorphisms of $N$ mod isotopy is isomorphic to $\oplus_{i=1}^n\left(Z \oplus Z\right)$, and every element extends over the cobordism $X$ in such a way that it is isotopic to the identity on $M$.  \\
{\bf Proof of claim:}  Let $P$ be the exterior of the link $L$, and write 
$$
\partial P = \T = \cup_{i=1}^n T_i
$$
for the boundary components of $P$. Then
$$
N = P \cup_{\T \times 0} \left(\T\times [0,1]\right) \cup_{\T \times 1} \left(S^3 - \nu(J)\right).
$$
Since both $P$ and $S^3 - \nu(J)$ are hyperbolic, the tori in the JSJ decomposition~\cite{jaco-shalen:seifert,johannson:book} of $N$ consist of the components of $\T$. It follows that any self-diffeomorphism of $f: N\to N$ is isotopic to one that preserves $\T\times [0,1]$.  Since the symmetry groups of $J$ and $P$ are trivial, it follows that we can assume that in fact $f$ is the identity on $\T \times 1$.   Moreover, $P$ and $S^3 - \nu(J)$ are not diffeomorphic, because they have a different number of boundary components, so $f$ must take $P$ to $P$ and $S^3 - \nu(J)$ to itself as well.
For each $i$, since $f$ on $T_i \times 0$ is homotopic to $f$ on $T_i \times 1$, we can assume that $f$ is the identity on $T_i \times 0$ as well.   According to Waldhausen~\cite{waldhausen:sufficiently-large}, the group of isotopy classes of diffeomorphisms of $T\times [0,1]$ (relative to the boundary) is isomorphic to the group of self-homotopy equivalences (again, relative to the boundary).  The latter is readily seen to be $\Z \oplus \Z$, where the elements can be described as follows. The element $(a,b) \in \Z \oplus \Z$ corresponds to the Dehn twist of  $S^1 \times S^1 \times [0,1]$ given by 
$$
(z,w,t) \to (e^{2\pi i at}z, e^{2\pi i bt}w, t).
$$

It follows trivially that the isotopy classes of diffeomorphisms of $\T\times [0,1]$ that preserve the components, relative to the boundary is a sum of copies of $\Z \oplus \Z$, with generators as described.   
Any such generator extends in a natural way over $S^1 \times D^2$. It is easy to see that the extension, as a diffeomorphism of $S^1 \times D^2$, is isotopic to the identity, via an isotopy that is the identity on the boundary.
For example, take the disk $D^2$ to have radius $2$, and write the diffeomorphism on $S^1 \times \{w\mid 1\leq |w| \leq 2\}$. So the extension over $S^1 \times D^2$ is given by 
$$
F(z,w) =
\begin{cases} 
(z,w)& \ \text{for }  |w| \leq 1\\
(e^{2\pi i a |w|}z, e^{2\pi i b|w|}w)& \ \text{for } 1 \leq |w| \leq 2.
\end{cases}
$$
Then the isotopy is given by 
$$
F_s(z,w) =
\begin{cases} 
(e^{2\pi i a s}z, e^{2\pi i bs}w)& \ \text{for }  |w| \leq 1\\
(e^{2\pi i as(2- |w|)}z, e^{2\pi i bs(2- |w|)}w)& \ \text{for } 1 \leq |w| \leq 2.
\end{cases}
$$

Write $V'$ for $X \cup_f W$, where the diffeomorphism $f$ is the restriction of $F$ to $M$.  (If $W$ is contractible, this is a cork twist along the embedded copy of $W$ in $V$.)  It has an obvious marking of the boundary coming from the identification of $N$ with a boundary component of $X$.  If $V'$ were diffeomorphic to $V$, preserving this marking, then we could glue this diffeomorphism to the identity of $\bar{X}$ to get a diffeomorphism
\begin{equation}\label{E:twist}
\bar{X} \cup_N X \cup_f W  \cong W
\end{equation}
But  $\bar{X} \cup_N X \cong N \times I$ (relative to the identity on the boundary) and hence $f$ extends to $\bar{X} \cup_N X$.  It follows that \eqref{E:twist} would result in a diffeomorphism of $(W,f)$ with $(W,\id)$, contradicting~\cite{akbulut:contractible}.
Since $V$ and $V'$ are simply connected homology balls, they are contractible, hence homeomorphic.  By the claim above and Lemma~\ref{L:extend}, there is no diffeomorphism between $V$ and $V'$.

Finally, the invertibility of the cobordism $X$ yields the embedding 
$$
V \subset V \cup_N \bar{X} = W \cup_M X \cup_n \bar{X} \cong  W. 
$$
Since $f$ extends naturally to a diffeomorphism of  $X \cup_N \bar{X} \cong M \times I$, the same argument produces an embedding of $V'$ in $W$.
\end{proof}
\begin{proof}[Proof of Theorem~\ref{T:main}]
Apply Theorem~\ref{T:general} to $(W,\tau)$, a cork or anti-cork (an exotic smoothing of a homotopy circle).  This gives a pair of non-diffeomorphic manifolds $V$ and $V'$ with the same boundary (contractible or a homotopy circle). But Freedman's theorem says that $V$ and $V'$ are in fact homeomorphic, so they can be viewed as (absolutely) exotic pairs.
\end{proof}

\section{Explicit examples of absolutely exotic corks and anti-corks}\label{S:examples}
In the proof of Theorem~\ref{T:general}, we used the fact that we can find a link $L$ in any $3$-manifold $M$ for which the symmetry group of $M-L$ is trivial. In practice, it is easy to find such links, and get some nice simple examples of absolutely exotic manifolds, discussed in this section. To prove that the examples have the required properties, we need to know something about the hyperbolicity and symmetry groups of certain manfolds; we used SnapPy and extensions as in Section~\ref{ss:ds} to rigorously verify these properties.  With the exception described in Example~\ref{E:alpha}, such assertions were rigorously checked in this way.  Curiously, we are unable to rigorously verify the simplest of the examples via such computer calculations.   

By \cite{akbulut1979mazur} the boundary of the Mazur cork $(W, \tau)$  (Figure~\ref{F:mazur}) is diffeomorphic to the $3$-manifold $M^{3}$, which is obtained by $+1$ surgery on the pretzel knot $K = P(-3,3,-3)$ of Figure~\ref{figa7}. 

\begin{figure}[ht]  
\begin{center}
\includegraphics[width=.4\textwidth]{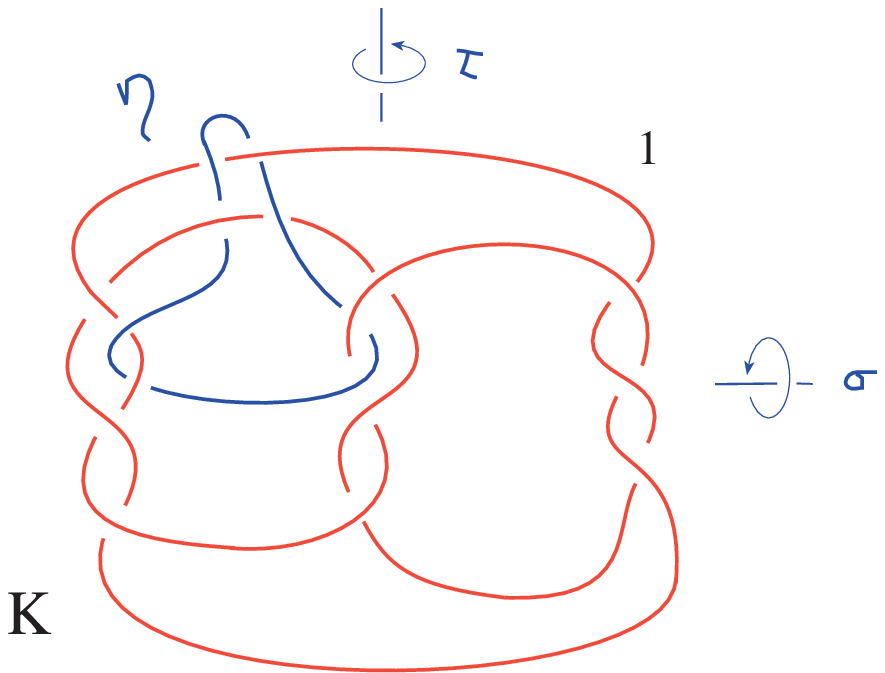}   
\caption{$K^{+1}$}   \label{figa7}
\end{center}
\end{figure}
\noindent

This diffeomorphism $f:\partial W\to \partial K^{+1}$ is explained by the steps of the  Figure~\ref{figa1}.   The curves $\alpha$ and $f(\alpha)$ will be used in the second example.

\begin{figure}[ht]  \begin{center}
\includegraphics[width=.75\textwidth]{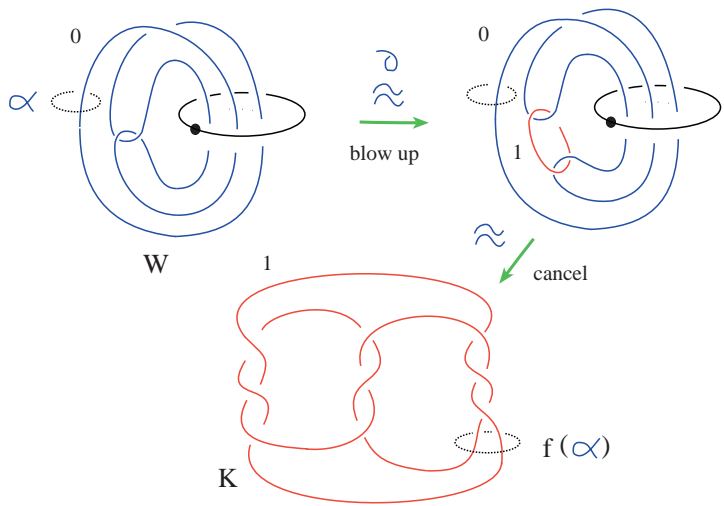}   
\caption{The diffeomorphism $f: \partial W\to  M$}   \label{figa1}
\end{center}
\end{figure}

The symmetry group of $K$ (up to isotopy) is a $Z_2 \oplus Z_2$ with generators $\sigma$ and $\tau$ as indicated in Figure~\ref{figa7}; both of these extend over the surgery to symmetries of $M$ which is,  by construction (see also~\cite{auckly-kim-melvin-ruberman:isotopy})
 $\tau$-equivariantly diffeomorphic to $\partial W$.

\begin{example}
Consider the knot $\eta$ in $M$, drawn in Figure~\ref{figa7}, and let $P$ be the exterior of a tubular neighborhood of $\eta$, with $T$  its boundary.  $P$ is obtained by doing $+1$ framed surgery on the first component of the link $L = K \cup \eta \subset S^{3}$ in the complement of a tubular neighborhood of $\eta $.  The choice of $\eta$ was made in order to disrupt the symmetries of $K$.


A patient reader can check that under the diffeomorphism $f:\partial W \to \partial K^{+1}$ of Figure~\ref{figa1}, $\eta$ corresponds to the curve $\eta$ in $\partial W$ indicated in Figure~\ref{figa2}.

\begin{figure}[ht]  \begin{center}
\includegraphics[width=.4\textwidth]{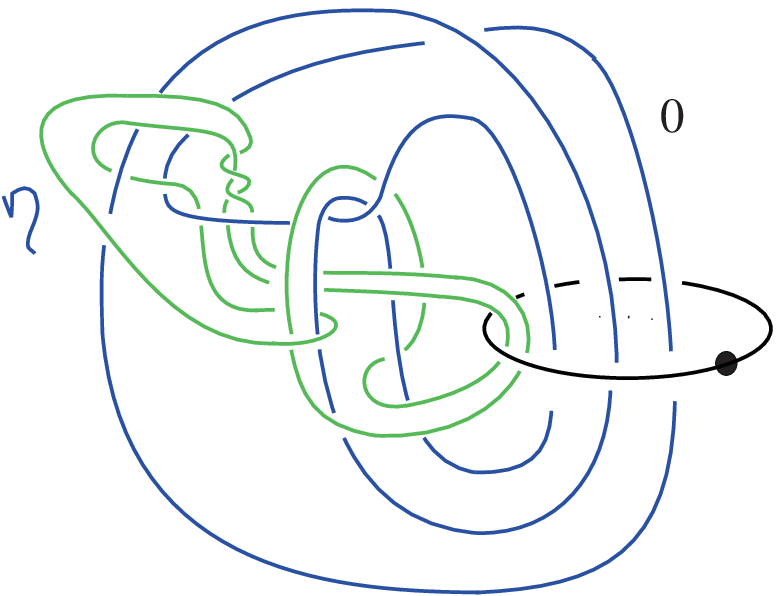}   
\caption{V}   \label{figa2}
\end{center}
\end{figure}

Now do the gluing construction as described in Lemma~\ref{L:ds}, using the knot $\eta$ for the link denoted $L$, and the concordance $C$ from the unknot to the doubly slice knot $11n42$ to get an invertible homology cobordism $X$ from $M$ to a $3$-manifold $N$.  Using the handle diagram in Figure~\ref{figa5}  for the complement of $C$ and Figure~\ref{figa2}, we see that this corresponds to the handlebody in Figure~\ref{figa8} (see 5.3 of~\cite{akbulut:book}). Hence Figure~\ref{figa8} describes $V$ with $\partial V=N$. Notice that 
In this figure  $W$ can easily be identified inside of $V$, and $V$ is built from $W$ by attaching two $1/2$-handle pairs (which is $X$). Also from Figure~\ref{figa8}, the reader can easily verify that $V$ is simply connected (as we saw more generally in the proof of Theorem~\ref{T:general}).

\begin{figure}[ht]  \begin{center}
\includegraphics[width=.75\textwidth]{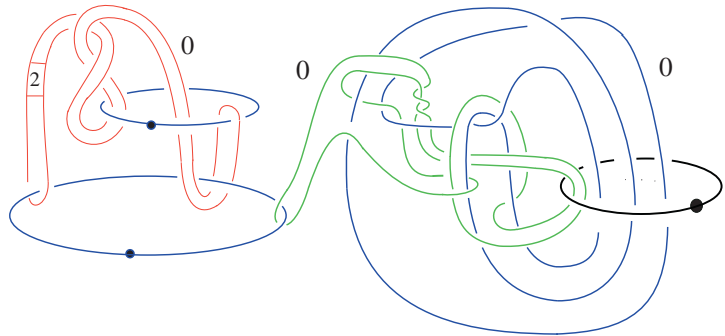}   
\caption{$W\subset V=W\cup_{M}X$ }   \label{figa8}
\end{center}
\end{figure}
\end{example}

\begin{example}\label{E:alpha}
We originally tried this construction making use of the simpler curve $\alpha$ in $M$ drawn in Figure~\ref{figa1}.  We computed using SnapPy that the corresponding manifold $P_1$ (resulting from $+1$ surgery on $K$) is hyperbolic with symmetry group $\Z_2$, generated by $\sigma$ as indicated in Figure~\ref{figa7}. However, the procedure of~\cite{dunfield-hoffman-licata:symmetry} for verifying this numerical calculation of the symmetry group breaks down for $P_1$. The reason, as explained to us by Dunfield, is that not all of the cells in the Epstein-Penner canonical cellulation~\cite{epstein-penner} of $P_1$ are tetrahedra. In this case we would also get a simpler $V_1$ as shown in Figure~\ref{figa9}.
If we assume that the symmetry group of $P_1$ is as stated, then a slightly more elaborate argument with the JSJ decomposition then implies that the corresponding manifold $N_1$ has trivial symmetry group.  This would imply  $V_1'$ is an absolutely exotic copy of $V_1$, but proving this would require a rigorous verification of the symmetry group.

\begin{figure}[ht]  \begin{center}
\includegraphics[width=.5\textwidth]{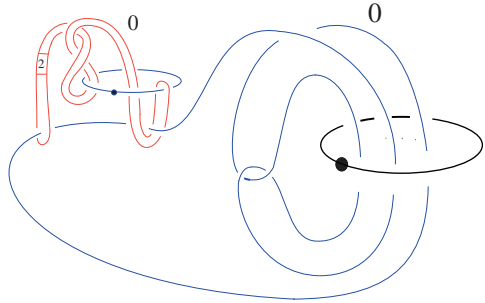}   
\caption{$W\subset V_1=W\cup_{M}X_1$ }   \label{figa9}
\end{center}
\end{figure}
\end{example}

\begin{example}\label{E:anticork}
As described in~\cite{akbulut:zeeman,akbulut:book,akbulut:stable} an anticork is a relatively fake manifold homotopy equivalent to $S^1$.  An example is given by the indicated involution of the boundary of the following manifold; the anticork itself is given by carving out the ribbon disk indicated by either of the dotted ribbon moves in Figure~\ref{F:anticork}.
\begin{figure}[htb]
\labellist
\small\hair 2pt
 \pinlabel {$\tau$}  at 140 222
\endlabellist
\centering
\includegraphics[scale=0.5]{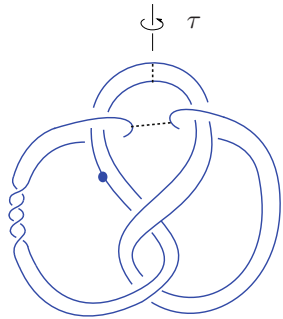}
\caption{An anticork $Q$}
\label{F:anticork}
\end{figure}

This can be turned into an absolutely exotic homotopy $S^1$ by the general technique from Theorem~\ref{T:general}.  As in the previous example, we can make this explicit and somewhat simpler by gluing in the complement of the $11n42$ slice along the knot $\eta \times I$ where $\eta$ is the curve indicated in~\ref{F:anti}.  That figure also shows the resulting absolute anticork.
\begin{figure}[ht]  
\begin{center}
\labellist
\small\hair 2pt
 \pinlabel {$\eta$}  at 175 80
\endlabellist
\includegraphics[width=.6\textwidth]{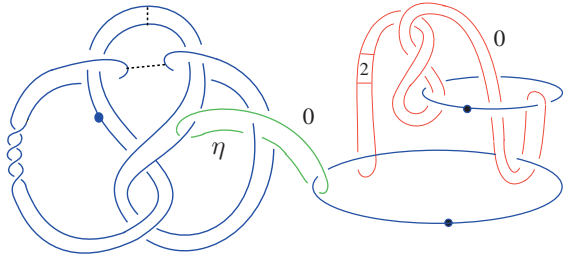}   
\caption{An absolute anticork}   \label{F:anti}
\end{center}
\end{figure}
\end{example}

\section{From infinitely many relative exotic structures  to absolutely exotic structures}\label{S:infinite}
In this section we will show how to modify a contractible manifold $W$ which admits infinitely many smooth structures relative to its boundary to infinitely many absolutely exotic smooth structures on a different contractible manifold $V$. The modification will not leave us with a full understanding of the symmetry group of the boundary, so we replace Lemma~\ref{L:extend} by a weaker result.
\begin{lemma}\label{L:finite}
Suppose that $M^n$ is a manifold such that $\pi_0(\diff(M))$ is finite.  If the manifold $V$ has infinitely many smoothings relative to some fixed identification $j: M \to \partial V$, then $V$ has infinitely many absolute smoothings.
\end{lemma}
\begin{proof}
Suppose that $V$ has only finitely many different smooth structures that are absolutely distinct, and let $V_k$, $k\in \N$ be infinitely many distinct smoothings relative to $j$ that are not diffeomorphic relative to $j$.  Then, replacing the $V_k$ by an appropriate subsequence, we may assume that all $V_k$ are diffeomorphic. Letting  $F_k:V_1 \to V_k$ be a diffeomorphism, we must have that  $j^{-1} \circ {F_k}|_{\partial V_k} \circ j$ are all distinct up to isotopy, contradicting our assumption that $\pi_0(\diff(M))$ is finite. 
\end{proof}

Now we make use of an old result of the second author~\cite{ruberman:seifert}, with a slight amplification.  (The argument could also be based on the construction of Paoluzzi-Porti, as in Theorem~\ref{T:general}.)
\begin{theorem}\label{T:invertible}
Let $M$ be a closed $3$ manifold. Then there is an invertible cobordism $X$ from $M$ to a hyperbolic manifold $N$, such that $\pi_1(M)$ normally generates $\pi_1(X)$.
\end{theorem}
\begin{proof}
All but the last clause is Theorem 2.6 of~\cite{ruberman:seifert}; the reader should beware that the labeling of boundary components $M$ and $N$ in that paper is reversed relative to this one. To see the last clause, we review the construction, introducing some new notation to lessen the confusion.  The main ingredient is an invertible tangle concordance from the complement of a trivial $g$-string tangle in the $3$-ball to a certain $g$-string tangle $T_g$.  The complement of the trivial tangle is a genus-$g$ handlebody $H_g$, and so the complement $X_g$ of this concordance is an invertible homology cobordism (relative to the boundary) from $H_g$ to $A_g$, the complement of the tangle $T_g$.  The main new observation is that the fundamental group of $X_g$ is normally generated by the meridians of the concordance, which are the same as the meridians of the trivial tangle.  In other words, the fundamental group of $X_g$ is the normal closure of $\pi_1(H_g)$.

For $g \geq 3$, the manifold $A_g$ has the property that when it is glued to itself by any diffeomorphism of the boundary surface, the result is a hyperbolic manifold.  Now, given a $3$-manifold $M$, we choose a Heegaard splitting of genus at least $3$, so that $M = H_g \cup_\phi H_g$. Then 
$$
X = X_g \cup_{}\phi \times \id_I
$$
is the required invertible homology cobordism $X$. It is straightforward to see that $\pi_1(M)$  normally generates $\pi_1(X)$.
\end{proof}

Now we have the main result of this section.
\begin{theorem}\label{T:infinite}
Suppose that $W$ is a contractible manifold, and let $M= \partial W$.  Suppose that $f_j: M \to M$ are diffeomorphisms that extend to homeomorphisms $F_j: W \to W$, giving infinitely many smoothings of $W$ relative to the identity. Then there is a contractible manifold $V$ with infinitely many smoothings.
\end{theorem}
\begin{proof}
We follow the proof of Theorem~\ref{T:main}.  Let $X$ be the invertible homology cobordism from $M$ to a hyperbolic manifold $N$ as in Theorem~\ref{T:invertible}, and set $V = W \cup_M X$.  Then we can form manifolds 
$$
V_j = W \cup_{f_j} X.
$$ 
Since $W$ is contractible, and $\pi_1(M)$ normally generates $\pi_1(X)$, it follows that the $V_j$ are all simply-connected and hence contractible.  Since they have the same boundary, they are all homeomorphic, but we claim that infinitely many of them are absolutely distinct smooth manifolds.  

As in the proof of Theorem~\ref{T:main}, the invertibility of $X$ implies that the $V_j$ are distinct smooth manifolds, relative to a fixed identification of $\partial V_j$ with $N$. But since $N$ is a hyperbolic manifold, $\pi_0(\diff(N))$ is finite~\cite{gabai:rigidity,gabai-meyerhoff-thurston} (a thorough discussion of such issues may be found in~\cite{hong-mccullough:survey}).  By Lemma~\ref{L:finite}, infinitely many of the  $V_j$ are absolutely distinct.
\end{proof}
Although their proofs are similar, Theorems~\ref{T:infinite} and~\ref{T:general} are logically independent. That is because the manifold $N$ used in the proof of Theorem~\ref{T:infinite} may well have a non-trivial symmetry group. For instance, there may be some symmetries derived from the $g$-fold symmetry of the hyperbolic manifold $A_g$.

\section{Homotopy equivalence and homotopy equivalence relative to the boundary}\label{S:homotopy}
It is a standard fact that for any $n \geq 4$ there are pairs of $n$-manifolds with boundary that are homotopy equivalent, have the same boundary, but are not homotopy equivalent relative to the boundary. As a simple example in dimension $4$, one can remove a $4$-ball from each of the two inequivalent $2$-sphere bundles over $S^2$. Both manifolds are homotopy equivalent to $S^2 \vee S^2$, but the cohomology rings (relative to the boundary) are not isomorphic.  Crossing with spheres gives examples in arbitrarily high dimensions. Of course the above $4$ dimensional example arises because the manifolds have different intersection forms;  in the case of $4$-manifolds having the same intersection form in \cite{boyer:relative} another subtle obstruction was found, and in  \cite {fan-wei:thesis} it was shown to occur in concrete examples. Here we note that the `infection' technique of Section~\ref{S:construct} gives more such examples in a direct way.  
\begin{theorem}\label{T:plug}
Suppose that $W$ is a compact $4$-manifold with $\partial W = M$ and that $\phi: M \to M$ is a homeomorphism that does not extend to a self-homotopy equivalence of $W$. Then $W$ contains a pair of  $4$-manifolds $V$ and $V'$ homotopy equivalent to $W$ with $\partial V \cong \partial V'$, such that there is no relative homotopy equivalence $(V,\partial V) \overset{\simeq}{\longrightarrow} (V', \partial V')$.
\end{theorem}
The prototypical example would be to take $W = S^2 \times D^2$, and to let $\phi: S^2 \times S^1\to S^2 \times S^1 $ be the Gluck twist, i.e., rotate the $S^2$ once around an axis as you go around $S^1$. Then $\phi$ does not even extend over $W$ as a map.  Related examples are {\em plugs}, which are certain Stein manifolds satisfying the hypotheses of Theorem~\ref{T:plug} (as corks they can be used to construct closed exotic manifolds); these were first described in~\cite{akbulut-yasui:corks-plugs}.  A systematic investigation of such manifolds is the subject of in the Michigan State Ph.D. thesis of Wei Fan~\cite{fan-wei:thesis}.
\begin{proof}[Proof of Theorem~\ref{T:plug}]
We proceed as in the proof of Theorem~\ref{T:general}, starting with the manifold $W$, and adding an invertible cobordism $X$ to obtain a manifold $V$ with boundary $N$. By construction, $N$ is obtained by gluing hyperbolic manifolds along incompressible tori, and so in particular is a Haken manifold.  Waldhausen's classical results~\cite{waldhausen:sufficiently-large} say that  that any self-homotopy equivalence of $N$ is homotopic to a homeomorphism.   It follows that any self-homotopy equivalence of $N$ extends over the cobordism $X$ so that it is the identity on $M$. The rest of the proof is as before; we make $V'$ by cutting out $W$ from $V$ and regluing via $\phi$. A homotopy equivalence between $V$ and $V'$, relative to the boundary, could then be modified to produce an extension of $\phi$ to a self-homotopy equivalence of $W$, contradicting our choice of $W$ and $\phi$. 
\end{proof}
\begin{example}
The simplest example that we can construct starts with $W = S^2 \times D^2$, with $\phi$ being the Gluck twist.  Then $S^2 \times S^1$ is invertibly homology cobordant to $N$, which is $0$-framed surgery on the Kinoshita-Terasaka knot.   SnapPy tells us that  $N$ is hyperbolic and has trivial symmetry group.  Unfortunately, the methodology that shows these calculations to be rigorous does not apply to closed manifolds, so this example cannot (yet) be verified rigorously.  
\end{example} 

\begin{example}
In order to proceed as in Theorem~\ref{T:plug} we need a knot $\eta$ in $S^1 \times S^2$ whose complement is hyperbolic and has trivial symmetry group.  An example comes from the $2$-component link denoted $9^2_{34}$ in Rolfsen's table~\cite{rolfsen:knots}. Note that both components of this link are unknotted, so that $0$-framed surgery on either component produces a knot in $S^2 \times S^1$. The extensions of SnapPy described above certify that both such knots are hyperbolic and their complements have trivial symmetry group.  (We came to this link via~\cite{henry-weeks}, where the symmetry group of the link complement was computed to be trivial.)  The $4$-manifold $V$ constructed using one of these components is drawn below in Figure~\ref{plug1}; the result of the Gluck twist, $V'$, is drawn in the second Figure~\ref{plug2}.
\end{example}



\begin{figure}[ht]
\centering
\begin{minipage}{.52\textwidth}
  \centering
  \includegraphics[width=.8\linewidth]{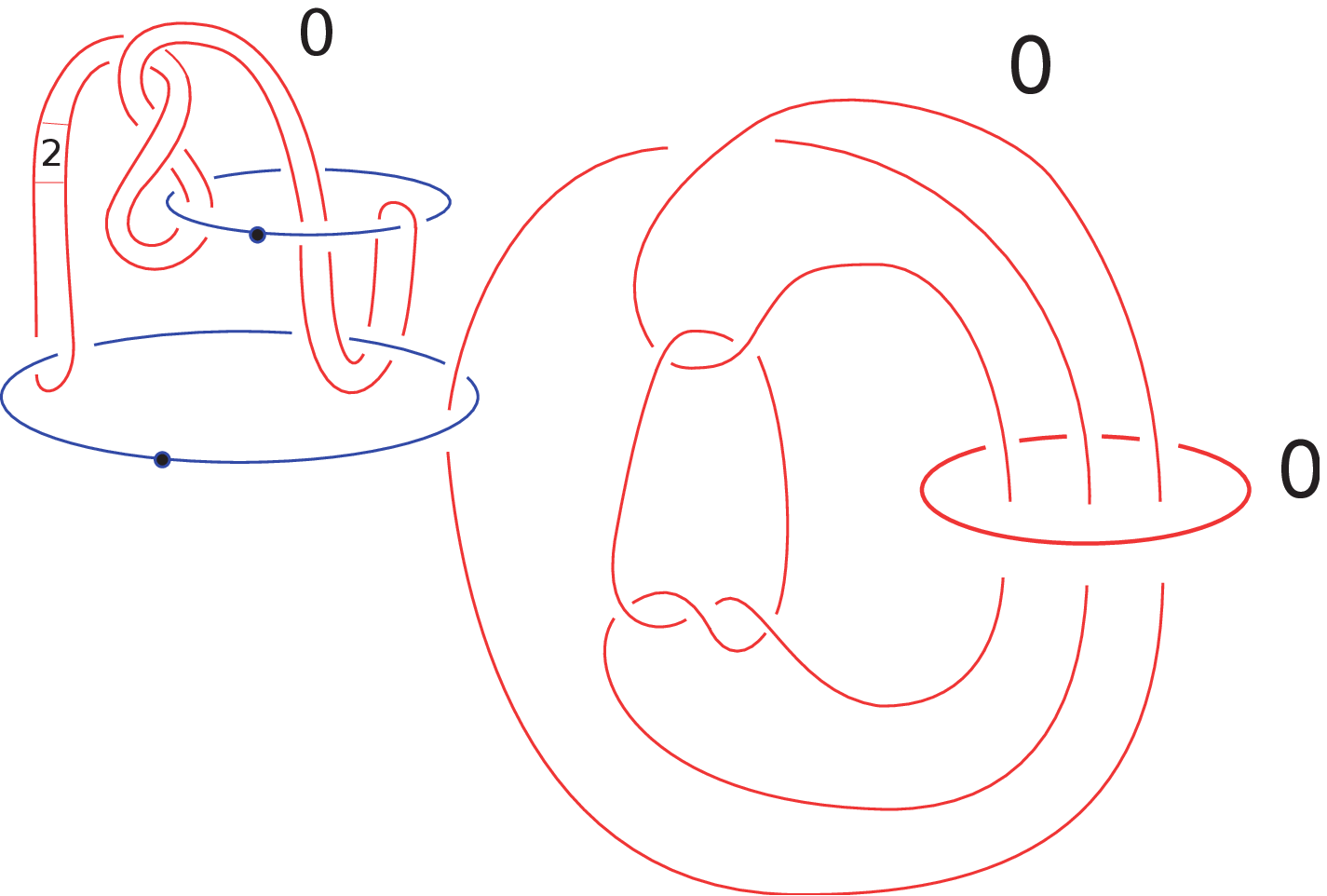}
  \caption{ $V = W \cup X$}
  \label{plug1}
\end{minipage}%
 \begin{minipage}{.52\textwidth}
  \includegraphics[width=.8\linewidth]{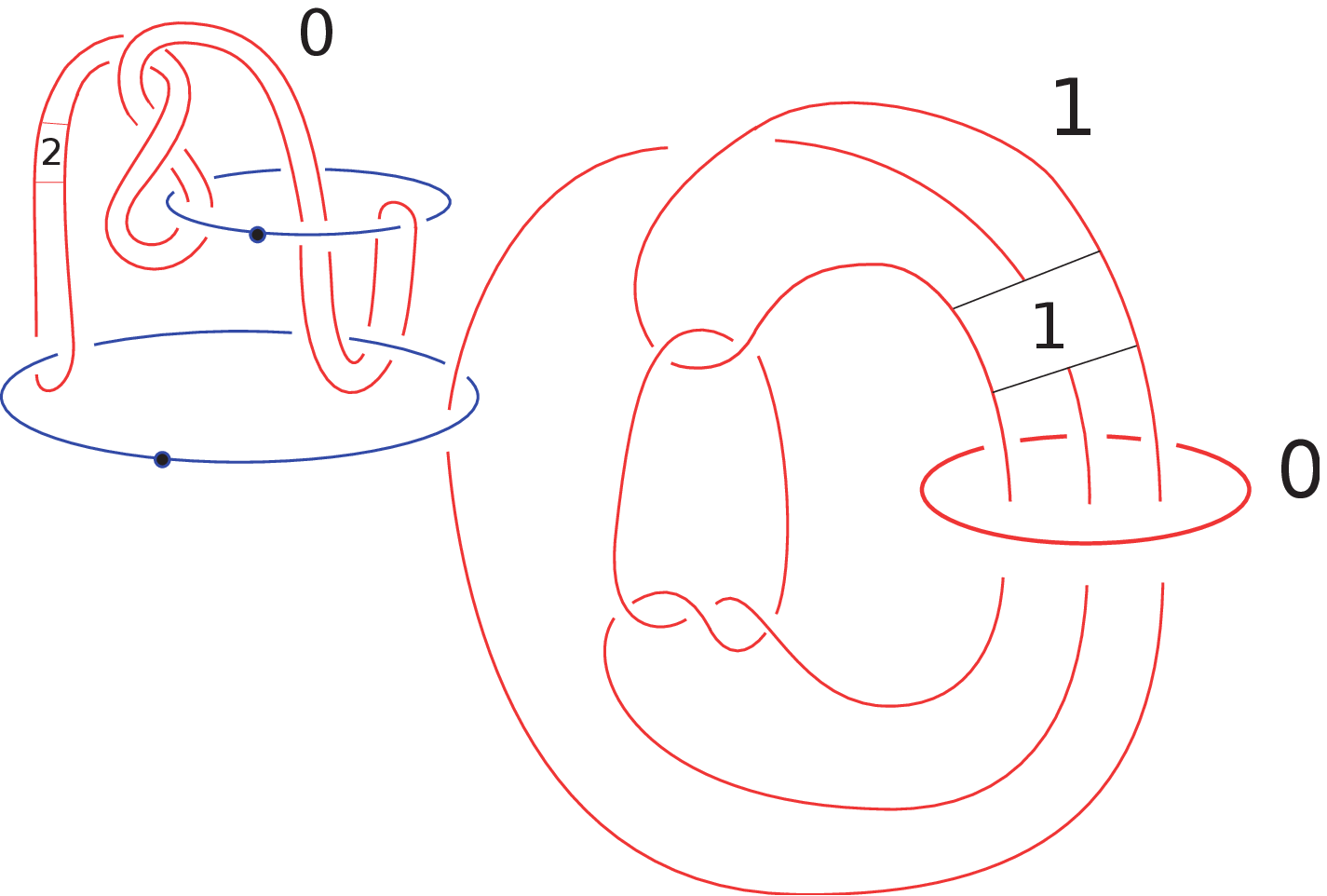}
  \caption{$V' = W \cup_\phi X$}
  \label{plug2}
\end{minipage}
\end{figure}

\newpage


\begin{thebibliography}{10}

\bibitem{akbulut1991exotic}
Selman Akbulut, \emph{An exotic 4-manifold}, Journal of Differential Geometry
  \textbf{33} (1991), no.~2, 357--361.

\bibitem{akbulut:contractible}
\bysame, \emph{A fake compact contractible {$4$}-manifold}, J. Differential
  Geom. \textbf{33} (1991), no.~2, 335--356.

\bibitem{akbulut:zeeman}
\bysame, \emph{A solution to a conjecture of {Z}eeman}, Topology \textbf{30}
  (1991), no.~3, 513--515.

\bibitem{akbulut:book}
\bysame, \emph{4-manifolds}, Book in preparation, available from\\
  \url{http://www.math.msu.edu/~akbulut/papers/akbulut.lec.pdf}, 2014.

\bibitem{akbulut:stable}
\leavevmode\vrule height 2pt depth -1.6pt width 23pt, {\em Isotoping 2-spheres
  in 4-manifolds}.
\newblock http://arxiv.org/abs/1406.5654, 2014.

\bibitem{akbulut1979mazur}
Selman Akbulut and Robion Kirby, \emph{Mazur manifolds.}, The Michigan
  Mathematical Journal \textbf{26} (1979), no.~3, 259--284.

\bibitem{akbulut-yasui:corks-plugs}
Selman Akbulut and Kouichi Yasui, {\em Corks, plugs and exotic structures}, J. G\"okova
  Geom. Topol. GGT, {\bf 2} (2008), 40--82.

\bibitem{auckly-kim-melvin-ruberman:isotopy}
Dave Auckly, Hee~Jung Kim, Paul Melvin, and Daniel Ruberman, \emph{Stable
  isotopy in four dimensions}, J. Lond. Math. Soc., to appear, \url{http://www.arxiv.org/abs/1406.4937}, 2014.
  
  \bibitem{boyer:relative}
  Steven Boyer, \emph{ Simply-connected 4-manifolds with a given boundary}, Trans. Amer. Math. Soc. 298, no.1 (1986) 331-357.


\bibitem{cerf:diffS3}
Jean Cerf, \emph{Sur les diff\'eomorphismes de la sph\`ere de dimension trois
  {$(\Gamma _{4}=0)$}}, Lecture Notes in Mathematics, No. 53, Springer-Verlag,
  Berlin-New York, 1968.

\bibitem{epstein-penner}
David B.~A. Epstein and Robert C. Penner, {\em Euclidean decompositions of noncompact
  hyperbolic manifolds}, J. Differential Geom., {\bf 27} (1988), 67--80.

\bibitem{knotinfo}
Jae~Choon Cha and Charles Livingston, \emph{Knotinfo table of knot invariants},
  \url{http://www.indiana.edu/~knotinfo/}, 2014.

\bibitem{SnapPy}
Marc Culler, Nathan~M. Dunfield, and Jeffrey~R. Weeks, \emph{Snap{P}y, a
  computer program for studying the topology of $3$-manifolds}, Available at
  \url{http://snappy.computop.org} (21/09/2014).

\bibitem{curtis-freedman-hsiang-stong}
Cynthia~L. Curtis, Michael~H. Freedman, Wu-Chung Hsiang, and Richard Stong,
  \emph{A decomposition theorem for {$h$}-cobordant smooth simply-connected
  compact {$4$}-manifolds}, Invent. Math. \textbf{123} (1996), no.~2, 343--348.

\bibitem{dunfield-hoffman-licata:symmetry}
Nathan~M. Dunfield, Neil~R. Hoffman, and Joan~E. Licata, \emph{Asymmetric
  hyperbolic {L}-spaces, {H}eegaard genus, and {D}ehn filling}, arXiv:1407.7827
  (2014).
  
\bibitem{fan-wei:thesis} 
Fan Wei, \emph{ Relative homotopy equivalences of smooth $4$-manifolds}, MSU thesis. 
 

\bibitem{gabai:rigidity}
David Gabai, {\em On the geometric and topological rigidity of hyperbolic
  {$3$}-manifolds}, J. Amer. Math. Soc., {\bf 10} (1997), 37--74.

\bibitem{gabai-meyerhoff-thurston}
David Gabai, G.~Robert Meyerhoff, and Nathaniel Thurston, {\em Homotopy hyperbolic
  3-manifolds are hyperbolic}, Ann. of Math. (2), {\bf 157} (2003), 335--431.

\bibitem{gluck:twist}
Herman Gluck, {\em The embedding of two-spheres in the four-sphere}, Trans. Amer.
  Math. Soc., {\bf 104} (1962), 308--333.
  
\bibitem{gordon:contractible}
Cameron McA. Gordon, \emph{Knots, homology spheres, and contractible
  {$4$}-manifolds}, Topology \textbf{14} (1975), 151--172.

\bibitem{henry-weeks}
Shawn R. Henry and Jeffrey R. Weeks, {\em Symmetry groups of hyperbolic knots and
  links}, J. Knot Theory Ramifications, {\bf 1} (1992), 185--201.

\bibitem{HIKMOT}
Neil Hoffman, Kazuhiro Ichihara, Masahide Kashiwagi, Hidetoshi Masai, Shin'ichi Oishi, and Akitoshi Takayasu,
  \emph{Verified computations for hyperbolic 3-manifolds}, 2013,
  arXiv:1310.3410. Code available from
  \url{http://www.oishi.info.waseda.ac.jp/~takayasu/hikmot/}.

\bibitem{hong-mccullough:survey}
Sungbok Hong and Darryl McCullough, {\em Mapping class groups of 3-manifolds, then and
  now}, in ``Geometry and topology down under'', vol.~597 of Contemp. Math.,
  Amer. Math. Soc., Providence, RI, 2013, 53--63.

\bibitem{jaco-shalen:seifert}
William Jaco and Peter Shalen, {\em Seifert fibered spaces in 3-manifolds}, Mem.\
  A.M.S.,  (1976).

\bibitem{johannson:book}
Klaus Johannson, {\em Homotopy equivalences of 3-manifolds with boundary}, Lecture
  Notes in Math., {\bf 761} (1976).

\bibitem{kinoshita-terasaka}
Shin'ichi Kinoshita and Hidetaka Terasaka, \emph{On unions of knots}, Osaka
  Math. J. \textbf{9} (1957), 131--153.

\bibitem{laudenbach-poenaru:handlebodies}
Fran{\c{c}}ois Laudenbach and Valentin Po{\'e}naru, \emph{A note on
  {$4$}-dimensional handlebodies}, Bull. Soc. Math. France \textbf{100} (1972),
  337--344.

\bibitem{lipshitz-ozsvath-thurston:bordered-hf}
Robert Lipshitz, Peter Ozsv{\'a}th, and Dylan Thurston, \emph{{Bordered
  Heegaard Floer homology: Invariance and pairing}},
  \url{http://arxiv.org/abs/0810.0687}, 2008.

\bibitem{matveyev:h-cobordism}
Rostislav Matveyev, \emph{A decomposition of smooth simply-connected
  {$h$}-cobordant {$4$}-manifolds}, J. Differential Geom. \textbf{44} (1996),
  no.~3, 571--582.

\bibitem{myers1}
Robert Myers, {\em Simple knots in compact, orientable 3-manifolds}, Trans.\
  A.M.S., {\bf 273} (1981), 75--92.

\bibitem{paoluzzi-porti:links}
Luisa Paoluzzi and Joan Porti, {\em Hyperbolic isometries versus symmetries of
  links}, Topology Appl., {\bf 156} (2009), 1140--1147.

\bibitem{rolfsen:knots}
Dale Rolfsen, ``Knots and Links'', Publish or Perish, Berkeley, 1976.

\bibitem{ruberman:seifert}
Daniel Ruberman, {\em Seifert surfaces of knots in ${S}^4$}, Pacific J.\ Math.,
  {\bf 145} (1990), 97--116.

\bibitem{sumners:inv1}
De Witt Sumners, {\em Invertible knot cobordisms}, in ``Topology of Manifolds'',
  J.~Cantrell and C.~Edwards, eds., Markham, 1970, 200--204.

\bibitem{sumners:inv2}
\leavevmode\vrule height 2pt depth -1.6pt width 23pt, {\em Invertible knot
  cobordisms}, Comm.\ Math.\ Helv., {\bf 46} (1971), 240--256.

\bibitem{waldhausen:sufficiently-large}
Friedrich Waldhausen, \emph{On irreducible 3-manifolds which are sufficiently large},
  Annals of Math. \textbf{87} (1968), 56--88.

\end{thebibliography}

\providecommand{\bysame}{\leavevmode\hbox to3em{\hrulefill}\thinspace}

\end{document}